\documentclass[12pt,draft]{article}

\usepackage{latexsym}
\usepackage{amsmath}
\usepackage{amsfonts}
\usepackage{amssymb}
\usepackage{enumerate}
\usepackage{cite}

\usepackage[latin1]{inputenc}

\numberwithin{equation}{section}

\newtheorem{theorem}{Theorem}[section]
\newtheorem{lemma}[theorem]{Lemma}
\newtheorem{proposition}[theorem]{Proposition}

\newtheorem{definition}[theorem]{Definition}
\newtheorem{example}[theorem]{Example}

\newcommand{\END}{\hfill \mbox{\raggedright $\Diamond$}}
\newcommand{\etal}{\emph{et al.}}
\newcommand{\ignore}[1]{}
\newcommand{\qed}{\hfill$\square$}

\newcommand{\N}{\mathbf{N}}

\newcommand{\R}{\mathbf{R}}

\newcommand{\C}{\mathbf{C}}

\newcommand{\SOtwo}{\mathbf{SO}(2)}

\newcommand{\GL}{\mathbf{GL}}
\newcommand{\Ttwo}{\mathbf{T}^2}

\newcommand{\Ztwo}{\mathbf{Z}_2}
\newcommand{\Zm}{\mathbf{Z}_m}
\newcommand{\Zthree}{\mathbf{Z}_3}

\newcommand{\IM}{\mathrm{Im}}

\newcommand{\Psigma}{\mathcal{P}_{\sigma}}
\newcommand{\Psigmaj}{\mathcal{P}_{\sigma^j}}
\newcommand{\Q}{\mathcal{Q}}

\newcommand{\J}{\{0, \ldots, m-1 \}}

\begin{document}

\title{Invariants and relative invariants \\ under compact Lie groups}

\maketitle

\begin{center}
\begin{tabular}{cc}
{\scshape Patrícia H. Baptistelli} &
{\scshape Miriam Manoel} \\
{\footnotesize {Departamento de Matem\'atica}} & {\footnotesize {Departamento de Matem\'atica, ICMC}} \\
{\footnotesize Universidade Estadual de
Maring\'{a}}  & {\footnotesize {Universidade de S\~ao Paulo}} \\
{\footnotesize {Av. Colombo, 5790, 87020-900}} &  {\footnotesize {13560-970 Caixa
Postal 668,}} \\
{\footnotesize { Maring\'{a}, PR - Brazil }} &  {\footnotesize{S\~ao Carlos, SP - Brazil }} \\
{\footnotesize {phbaptistelli@uem.br }} &  {\footnotesize{miriam@Icmc.usp.br}}

\end{tabular}
\end{center}

\begin{abstract}
This paper presents algebraic methods for the study of polynomial relative invariants,
when the group $\Gamma$ formed by the symmetries and relative symmetries is a compact
Lie group. We deal with the case when the subgroup $H$ of symmetries is normal in $\Gamma$ with index $m,$ $m \geq 2.$
For this, we develop the invariant theory of
compact Lie groups acting on complex vector spaces. This is the starting point for the study of relative invariants and the computation of their generators. We first obtain the space of the invariants under the subgroup $H$ of $\Gamma$ as a direct sum of $m$ submodules over the ring of invariants under the whole group. Then, based on this decomposition, we construct a Hilbert basis of the ring of $\Gamma$-invariants from a Hilbert basis of the ring of $H$-invariants. In both results the knowledge of the relative Reynolds operators defined on $H$-invariants is shown to be an essential tool to obtain the invariants under the whole group. The theory is illustrated with some examples.

\end{abstract}



\section{Introduction}

Many natural phenomena  possess symmetry properties.  The equations describing a physical or biological system may have symmetries as a result of its geometry, modeling assumptions and simplifying normal forms. So when this occurrence is taken into account in the
mathematical formulation of the problem, it can simplify significantly its interpretation. The natural language for describing symmetry properties of a dynamical system is that of group theory. The main point
is that, in this case, the set $\Gamma$ formed by these symmetries has structure of group and the differential equations governing the system
remain unchanged under the action of $\Gamma.$ More specifically, consider the system
\begin{equation} \label{eq:INTRO1}
 \dot{x}~=~G(x)
\end{equation}
defined on a finite-dimensional vector space $V$ of state
variables, where $G:V\to V$ is a smooth vector field. We assume
that $\Gamma$ is a compact Lie group acting orthogonally
on $V.$ If $G$ is {\it purely $\Gamma-$equivariant}, i.e.,
\begin{equation}
\label{equiv}
 G(\gamma x)~=~\gamma \, G(x)~, \quad \forall \gamma \in \Gamma, \ x \in V,
\end{equation}
then $x(t)$ is a solution of (\ref{eq:INTRO1}) if, and only if,
$\gamma x(t)$ is a solution of (\ref{eq:INTRO1}), for all $\gamma \in \Gamma.$
The effect of symmetries in local and global dynamics has been
investigated for decades by a great number of authors in many
papers (an extensive list of references can be found in \cite{chossat})  .

There is a parallel class of equivariant dynamical systems, in the sense of \eqref{equiv}, which are those in simultaneous presence of symmetries and reversing symmetries. A system is said to have a reversing symmetry if the equations
of motion are invariant under that symmetry in combination with the inversion
of time. This contrasts with a symmetry,
which leaves the equations of motion invariant without time inversion. In algebraic terms, this corresponds to considering a group epimorphism
\begin{equation}
\label{sigma}
\sigma:\Gamma\to\Ztwo,
\end{equation}
where $\Ztwo$ denotes the multiplicative group $\{1, -1\}$, whose kernel is the subgroup $\Gamma_+$ of the symmetries in $\Gamma$. The complement $\Gamma_- =
\Gamma\setminus\Gamma_+$ is the set of reversing symmetries. The mapping $G: V \rightarrow V$ is called {\it reversible-equivariant} if
 \begin{equation}
\label{rev-equiv}
 G(\gamma x)~=~\sigma(\gamma)\gamma \, G(x)~, \quad \forall \gamma \in \Gamma, \ x \in V.
\end{equation}
Then $x(t)$ is a solution of (\ref{eq:INTRO1}) if, and only if, $\gamma x(\sigma(\gamma)t)$ is a solution of
(\ref{eq:INTRO1}), for all $\gamma \in \Gamma.$

In \cite{BM1} we have developed an approach based on singularity theory in combination with group representation theory, in the
same lines as Golubitsky \etal~\cite{GS69}, for the systematic analysis of reversible-equivariant steady-state bifurcations.
More recently, we have used in \cite{BM3} techniques of group representation theory to describe the decomposition $\sigma$-isotypic in the reversible-quivariant context, which is the corresponding isotypic decomposition presented in \cite{GS69} in the equivariant context.

The starting point for local or global analysis of reversible-equivariant
systems is to find the general form of the vector field $G$
in \eqref{eq:INTRO1} that satisfies (\ref{rev-equiv}). The theorems by Schwarz and Po\'enaru (see
Golubitsky \etal~\cite{GS69}) reduce this task to a purely
algebraic problem in invariant theory. In \cite{antoneli}, we have established
algebraic results that reduce the search for the general form of the vector field to the polynomial invariant case. Based on a link existent between the invariant theory for $\Gamma$ and for its
subgroup $\Gamma_+,$ we have provided a set of generators for the modules of relative invariants and reversible-equivariants, from generators of the ring of invariants under the action of $\Gamma_+$ on $V.$ We remark that these, in turn, are well known in several important cases, as in \cite{Gat,sturm,Worf}. More specifically, consider an epimorphism $\sigma$ as in (\ref{sigma}). A polynomial function ${f}:V \to \R$ is called   \emph{anti-invariant} if
\begin{equation} \label{eq:ANTIINV}
{f}(\gamma x)~=~\sigma(\gamma) {f}(x)~,
\end{equation}
for all $\gamma \in \Gamma$ and $x \in V$. If $\sigma$ in (\ref{sigma}) were the trivial homomorphism, the function $f$ satisfying (\ref{eq:ANTIINV}) would reduce to a \emph{$\Gamma$-invariant} function, or, as we shall
simply say,  an {\it invariant} function.
The space $\Q(\Gamma)$ of all
anti-invariant polynomial functions is a finitely generated graded module over the ring
$\mathcal{P}(\Gamma)$ of invariant polynomial functions on
$V.$ The ring $\mathcal{P}(\Gamma)$ in turn  admits a Hilbert basis (namely, a finite set of generators), once $\Gamma$ is compact,
according to the Hilbert-Weyl Theorem (see \cite[Theorem XII 4.2]{GS69}).
 In \cite{antoneli} we have obtained the decomposition
\begin{equation} \label{eq:decomposition1}
 \mathcal{P}(\Gamma_{+})~=~\mathcal{P}(\Gamma)
 \oplus \mathcal{Q}(\Gamma)
 \end{equation}
as a direct sum of modules over $\mathcal{P}(\Gamma)$. This decomposition is the basis for a second result,
namely the algorithm that computes generating sets for
anti-invariant polynomials  as a module over the ring of the invariant
polynomials.

In this paper we extend the formalism described above when the set of the symmetries $H$ is a normal
subgroup of index $m$ greater than 2. The functions now are complex-valued and the representation of
$\Gamma$ is a complex representation. We remark that when $m=2$ the existence of the epimorphism (\ref{sigma})
is a natural consequence of the fact that $\Gamma_+$ to have index 2. However, not for all $m \in {\Bbb R}$
there is an epimorphism from $\Gamma$ to $\Zm,$ i.e., the existence of a normal subgroup of arbitrary finite
index with cyclic quotient is not guaranteed. So this is an assumption from now on. Let us just  recall that it does exist if $m$ is prime, for example. Hence, we consider an epimorphism
\begin{equation} \label{eq:HOMOSIGMA}
 \sigma : \Gamma \to \Zm,
\end{equation}
where $\Zm$ denotes the cyclic group generated by the $m$-th root of unity, with $H = \ker(\sigma).$ As a consequence, if we fix $\delta \in \Gamma$ such that $\sigma(\delta) $ is a primitive $m$-th root of unity,
then we have a decomposition of $\Gamma$ as a disjoint union of  left-cosets:
\begin{equation}
\label{eq:left-cosets}
\Gamma = H \ \dot{\cup} \ \delta \ H \ \dot{\cup} \ \delta^2 \ H  \ \dot{\cup} \ \ldots \ \dot{\cup} \ \delta^{m-1} \ H.
\end{equation}

In section 2, we present our first main result, Theorem~\ref{thm:RELATIVEDECOMP},
 which decomposes $\mathcal{P}(H)$
as a direct sum of $m$ submodules over $\mathcal{P}(\Gamma)$. This extends the decomposition
(\ref{eq:decomposition1}) when $m=2$ for the cases  $m>2$. This also extends the analogous result
obtained by
Smith in  \cite{larry} in the case that the group is finite for compact Lie groups.

In section 3 we obtain the other two main results, related to the construction of a Hilbert basis
for the ring $\mathcal{P}(\Gamma)$.
One is  Theorem~\ref{thm:MAIN1} for $m=2$ and the other is
Theorem~\ref{thm:MAIN2} for $m>2$, which are  applications of the decomposition
given in section 2. The result is to construct a Hilbert basis of the ring $\mathcal{P}(\Gamma)$ from a
Hilbert basis of the ring $\mathcal{P}(H)$.
The main idea of this construction is to make use of relative Reynolds operators defined on  $\mathcal{P}(H)$ that project onto each of its $m$ components. We point out that the interest of this
result is that from the knowledge of the invariants just under the symmetries in $\Gamma$ we can recover
the generators of the invariants under the whole group.

\section{The Structure of the $\sigma$-Relative Invariants} \label{sec:RELATIVEINV}
In this section we give some basic concepts and obtain
results about the structure of the set of relative invariant
functions.

\subsection{The Representation Theory}
Let $\Gamma$ be a compact Lie group acting linearly on a
finite-dimensional vector space $V$ over $\C.$ This action corresponds to
a representation $\rho$ of the group $\Gamma$ on $V$ through a
linear homomorphism from $\Gamma$ to the group $\GL(V)$ of
invertible linear transformations on $V$. We denote by $(\rho,V)$
the vector space $V$ together with the linear action of $\Gamma$
induced by $\rho$. We often write $\gamma x$ for the action of the
linear transformation $\rho(\gamma)$ of an element $\gamma \in
\Gamma$ on a vector $x \in V$, whenever there is no danger of
confusion.

Consider an epimorphism $\sigma : \Gamma \to \Zm.$ We notice that $\sigma$ defines a one-dimensional representation of
$\Gamma$ on $\C$, by $z \mapsto \sigma(\gamma)z$ for $z \in \C$.
Also, the character afforded by  $\sigma$ is the complex-valued function whose value on
$\gamma \in \Gamma$ is $\sigma(\gamma)$; in other words, it is
$\sigma$ itself, but regarded as a complex-valued function, which we
shall also denote by $\sigma$.

\begin{definition}
For each $j \in \J,$ a polynomial function $f: V \to \C$ is called
$\sigma^j$-relative invariant if
\begin{equation} \label{eq:SIGMARELATIVEINV}
f(\gamma x)~=~\sigma^j(\gamma) f(x)
\end{equation}
for all $\gamma \in \Gamma$ and $x \in V$.  \END
\end{definition}

The set of the $\sigma^j$-relative invariant polynomial functions is a module over
$\mathcal{P}(\Gamma)$ and we shall denoted it by $\Psigmaj(\Gamma)$. Notice that we can regard a $\sigma^j$-relative invariant polynomial function on $V$ as an equivariant mapping  from $(\rho,V)$ to $(\sigma^j,\C)$. Therefore, the existence of a finite generating set
for $\Psigmaj(\Gamma)$ is guaranteed by Po\'enaru's Theorem (see \cite[Theorem XII 6.8]{GS69}).

The next lemma gives a relation between the $\sigma$-relative invariants and the invariants under the action of the subgroup $H = \ker \sigma.$ This result is immediate from the definitions.

\begin{lemma}
\label{lem:CARACTERIZATION}
Let $\sigma$ be an epimorphism as in \eqref{eq:HOMOSIGMA}, with kernel $H,$ and fix $\delta \in \Gamma$ such that $\sigma(\delta)$ is the primitive $m$-th root of unity. Then for each $j \in \J,$ we have
\[ \Psigmaj(\Gamma)~=~\{f \in \mathcal{P}(H):
 f(\delta x) = \sigma^j(\delta)f(x) \,,\;\forall\: x \in V\}.\]
\end{lemma}

\subsection{The Application of the Invariant Theory}

The aim here is to relate the rings of invariants under the action of $\Gamma$ and $H.$  The modules of
$\sigma^j$-relative invariants presented in the previous subsection play a fundamental role in this construction. The result is Theorem~\ref{thm:RELATIVEDECOMP}, with which we
end the subsection.

We start by introducing the \emph{$\sigma^j$-relative Reynolds operators} on
$\mathcal{P}(H)$,
$R_j\!:\mathcal{P}(H) \to
\mathcal{P}(H)$, by
\[
 R_j(f)(x)~=~\frac{1}{m}
 \sum_{\;\gamma H} \overline{\sigma^j(\gamma)} f(\gamma x),\] for each $j \in \J$, where bar means the complex conjugation. So
\begin{equation}
 \begin{split} \label{eq:REYNOLDS2}
 R_j(f)(x)~
 & =~\frac{1}{m}
 \big(f(x)+\overline{\sigma^j(\delta)}f(\delta x) + \ldots + \overline{\sigma^j(\delta^{m-1})}f(\delta^{m-1} x)  \big)\\
 & = \frac{1}{m}
 \sum_{k=0}^{m-1} \overline{\sigma^{jk}(\delta)} f(\delta^k x),
 \end{split}
\end{equation}
for an arbitrary (but fixed) $\delta \in \Gamma$ such that $\sigma(\delta)$ is the primitive $m$-th root of unity.

We notice that for $m=2,$ the Reynolds operator $R^\Gamma_{\Gamma_+}$ and the Reynolds
$\sigma$-operator $S^\Gamma_{\Gamma_+}$ defined in \cite{antoneli} correspond to $R_0$ and $R_1,$ respectively. There, the operator $R_1$ has been used to produce a generating set for the anti-invariant polynomial functions as a module over the ring of the invariants.

For the next proposition we recall that, for $\delta \in \Gamma \setminus H$ such that $\delta^m \in H, $
\begin{equation} \label{eq:ROOTUNITY}
\sum_{i = 0}^{m-1}\sigma^{ik}(\delta)= 1 + \sigma^k(\delta) + \sigma^{2k}(\delta) + \ldots + \sigma^{(m-1)k}(\delta) = 0,
\end{equation}
for all $k \in \{1, \ldots, m-1\}.$

Now let us denote by $I_{\mathcal{P}(H)}$ the identity map on $\mathcal{P}(H)$.  We then have:

\begin{proposition} \label{prop:PROPERTIES1}
For each $j \in \J,$ the Reynolds operators $R_j$ satisfy the
following properties:
\begin{enumerate}[(i)]
 \item They are homomorphisms of $\mathcal{P}(\Gamma)$-modules and
\begin{equation} \label{eq:MAPSUM1}
        R_0 + R_1 + \ldots + R_{m-1}~=~I_{\mathcal{P}(H)}~.
\end{equation}
 \item They are idempotent projections, with $\IM(R_j)~=~\Psigmaj(\Gamma).$
 \item For all $1 \leq j \leq m-1,$ we have \[\mathcal{P}_{\sigma^j}(\Gamma) \cap \bigl(\mathcal{P}(\Gamma) + \ldots + \mathcal{P}_{\sigma^{j-1}}(\Gamma)\bigr) = \{0\}.\]
 \end{enumerate}
\end{proposition}
\begin{proof}
To prove {\it({i})} we just note that if $h,g \in
\mathcal{P}(H)$ and $f \in \mathcal{P}(\Gamma)$, then
\[
 R_j(h+g)~=~R_j(h) + R_j(g)
 \qquad\text{and}\qquad
 R_j(f\,h)~=~f\,R_j(h),
\] for all $j \in \J.$
Moreover,  (\ref{eq:MAPSUM1})  follows directly from
(\ref{eq:REYNOLDS2}) and (\ref{eq:ROOTUNITY}).

To prove {\it (ii)} and {\it (iii)}, we fix $\delta \in \Gamma$ such that $\sigma(\delta)$ is the primitive $m$-th root of unity. For any $f \in\mathcal{P}(H)$, it is direct from (\ref{eq:REYNOLDS2}) and Lemma \ref{lem:CARACTERIZATION} that
$R_j(f) \in \Psigmaj(\Gamma).$  Hence,  $\IM(R_j) \subseteq \Psigmaj(\Gamma).$
The other inclusion is immediate, since $R_j |_{\Psigmaj(\Gamma)}=I_{\Psigmaj(\Gamma)}.$ From these, it is straightforward that
$R_j$ is idempotent.

We prove {\it (iii)} by induction on $j.$ It is easy to see that $\mathcal{P}_{\sigma}(\Gamma) \cap \mathcal{P}(\Gamma) = \{0\}.$ Suppose now that given $1 \leq s \leq m-2,$ we have
\begin{equation} \label{eq:INDUCTION}
\mathcal{P}_{\sigma^j}(\Gamma) \cap \bigl(\mathcal{P}(\Gamma) + \ldots + \mathcal{P}_{\sigma^{j-1}}(\Gamma)\bigr) = \{0\},
\end{equation}
for all $1 \leq j \leq s.$ We show that (\ref{eq:INDUCTION}) is true for $s+1.$ In fact, let $f \in \mathcal{P}_{\sigma^{s+1}}(\Gamma) \cap \bigl(\mathcal{P}(\Gamma) + \ldots + \mathcal{P}_{\sigma^s}(\Gamma)\bigr).$ Then
\[f(\delta x) = \sigma^{s+1}(\delta)f(x) \qquad \text{and} \qquad f(x) = h_0(x) + \ldots + h_s(x),\] with $h_j \in \Psigmaj(\Gamma).$ So
\[(\sigma^{s+1}(\delta) - 1)h_0(x) + \ldots + (\sigma^{s+1}(\delta) - \sigma^s(\delta))h_s(x) = 0,\] for all $x \in V.$ By the induction hypothesis $(\sigma^{s+1}(\delta) - \sigma^j(\delta))h_j(x) = 0,$ for all $0 \leq j \leq s.$ It follows that $h_j \equiv 0,$ for all $0 \leq j \leq s,$ as desired. \qed
\end{proof}\\

We are now in position to state the following theorem, our first main result, which is now an immediate consequence of the proposition above.

\begin{theorem} \label{thm:RELATIVEDECOMP}
The following direct sum decomposition of $\mathcal{P}(\Gamma)$-modules holds:
\[
 \mathcal{P}(H)~=~\mathcal{P}(\Gamma) \oplus \Psigma(\Gamma)~\oplus~\ldots~\oplus \mathcal{P}_{\sigma^{m-1}}(\Gamma).
\]
\end{theorem}

\section{Hilbert basis for the $\Gamma$-invariants}

\label{sec:ALGRELATINV}

In  this section we show how to compute generators for $\mathcal{P}(\Gamma),$ as a ring, from a Hilbert basis of $\mathcal{P}(H),$ where $H$ is a normal subgroup of index $m$ of $\Gamma.$ The basic idea here is to take advantage of the direct sum decomposition from Theorem \ref{thm:RELATIVEDECOMP} to transfer generating sets from one ring to the other. The procedure is algorithmic and it holds, as its proof, for
any $m$. However, the case $m=2$ is special in the sense that a distinguished proof reveals a class of
invariants and anti-invariants that does not appear in the proof for $m>2$.  For this reason, we present
the two cases separately, as Theorems \ref{thm:MAIN1} and \ref{thm:MAIN2}.

We start with index $m=2.$ For this case, if $u: V \rightarrow \R$ and $\gamma \in \Gamma,$ we denote by $\gamma u: V \rightarrow \R$ the polynomial function defined by
\[\gamma u (x) : = u( \gamma x), \quad \forall x \in V.\]

\begin{lemma} \label{lemma:MAIN1}
Let $u: V \rightarrow \R$ be a real-valued function. Then for all $t \in \N,$
\begin{equation}
\label{eq:PROPERTIE1}
u^t + \delta u^t = F_t\bigl(R_0(u),(R_1(u))^2\bigr),
\end{equation}
\begin{equation}
\label{eq:PROPERTIE2}
\sum_{k=0}^{t}u^k \delta u^{t-k} = G_t\bigl(R_0(u),(R_1(u))^2\bigr),\end{equation}
and
\begin{equation}
\label{eq:PROPERTIE3}
u^t - \delta u^t = R_1(u)H_t\bigl(R_0(u),(R_1(u))^2\bigr),
\end{equation}

\noindent for $F_t,G_t,H_t: \R^2 \rightarrow \R$ polynomial functions.
\end{lemma}

\begin{proof} We check (\ref{eq:PROPERTIE1}) and (\ref{eq:PROPERTIE2}) by induction on $t,$ by using the polynomial
identity
\[
 u^i + \delta u^i~=~(u + \delta u)(u^{i-1} + \delta u^{i-1}) - (u \ \delta u)(u^{i-2} + \delta u^{i-2})\] in the first case and the identity
 \[\sum_{k=0}^{i}u^k \delta u^{i-k} = (u^i + \delta u^i) +
(u \ \delta u )\sum_{k=0}^{i-2}u^k \delta u^{i-2-k}\] in the second case.

To prove (\ref{eq:PROPERTIE3}) just write
\[u^t - \delta u^t = (u - \delta u)\bigl(\sum_{k=0}^{t-1}u^k \delta u^{t-1-k}\bigr)\] and use (\ref{eq:PROPERTIE2}). \qed
\end{proof}

\begin{theorem}
\label{thm:MAIN1}
Let $\sigma: \Gamma \rightarrow \Ztwo$ be an epimorphism. Let $\{u_1,\ldots,u_s\}$ be a
Hilbert basis of the ring $\mathcal{P}(\Gamma_{+}),$ where $\Gamma_{+}$ is the kernel of $\sigma$.
Then the set
\[\{R_0(u_i), R_1(u_i)R_1(u_j), \ 1 \leq i,j \leq s \}\]
forms a Hilbert basis of the ring $\mathcal{P}(\Gamma).$
\end{theorem}

\begin{proof}
Let us fix $\delta \in \Gamma  \setminus \Gamma_{+}$ and let $\{u_1,\dots,u_s\}$ be a
Hilbert basis of $\mathcal{P}(H)$. We need to show that every polynomial function
$f \in\mathcal{P}(\Gamma)$ can be written as
\[f = h\bigl(R_0(u_i), R_1(u_i)R_1(u_j), \ 1 \leq i,j \leq s\bigr),\] where $h: \R^{s(s+3)/2} \rightarrow \R$ is a polynomial function. The proof is done by induction
on the cardinality $s$ of the set $\{u_1, \dots, u_s\}$.

Let $f \in \mathcal{P}(\Gamma).$ From Proposition~\ref{prop:PROPERTIES1}, there exists $\tilde{f}
\in\mathcal{P}(\Gamma_+)$ such that
$R_0(\tilde{f})=f$. First we write
\[ \tilde{f}(x)~=~\sum_{\alpha} a_\alpha \, u^\alpha(x)~,
\]
with $a_\alpha\in\R$, where $\alpha= (\alpha_1, \ldots,
\alpha_s) \in \N^s$, and $u^\alpha~=~u_1^{\alpha_1}\ldots u_s^{\alpha_s}$. Then
\[
 f~=~\sum_{\alpha} a_\alpha \,
 (u^\alpha + \delta u^\alpha)~.
\]
Assume $s=1$. Then we may write $f~=~\sum_{i} a_i \, (u^i + \delta u^i)~,$
where $a_i \in \R$ and $i \in\N$. Now we use (\ref{eq:PROPERTIE1}) to get
\[
 f~=~\sum_i a_i \, h_i\bigl(R_0(u), (R_1(u))^2\bigr)~=~h\bigl(R_0(u), (R_1(u))^2\bigr)~,
\] with $h:\R^2 \rightarrow \R$ a polynomial function.

Assume that given $l \in \N,$ for all sets $\{u_1, \ldots, u_l\}$ with $1 \leq l \leq s-1$ we have
\begin{equation} \label{induction1}
 \sum_\alpha a_\alpha \, (u^\alpha
+ \delta u^\alpha)~=~h_{\alpha}\bigl(R_0(u_i), R_1(u_i)R_1(u_j), \ 1 \leq i,j \leq l\bigr),
\end{equation} where $a_i \in \R,$ $\alpha \in\N^l$ and $h_{\alpha} : \R^{l(l+3)/2} \rightarrow \R$ is a polynomial function. Consider the set $\{u_1, \ldots, u_l\} \cup \{u_{l+1}\}$ and let $f \in \mathcal{P}(\Gamma).$ Then we may write
\[
\begin{split}
f
 & = \!\!\sum_{\alpha,\alpha_{\ell+1}} \!\!
       a_{\alpha,\alpha_{\ell+1}}
       (u^\alpha u_{\ell+1}^{\alpha_{\ell+1}}
       +\delta u^\alpha \delta u_{\ell+1}^{\alpha_{\ell+1}}) \\
 & =  \frac{1}{2} \bigl [ \!\!\sum_{\alpha,\alpha_{\ell+1}} \!\!
      a_{\alpha,\alpha_{\ell+1}} \!
      (u^{\alpha}
      + \delta u^{\alpha}) (u_{\ell+1}^{\alpha_{\ell+1}}
      + \delta u_{\ell+1}^{\alpha_{\ell+1}}) \ + \\ &
      \hspace*{6.4cm} + \!\!\sum_{\alpha,\alpha_{\ell+1}} \!\!
      a_{\alpha,\alpha_{\ell+1}} \! (u^{\alpha}
      - \delta u^{\alpha}) (u_{\ell+1}^{\alpha_{\ell+1}}
      - \delta u_{\ell+1}^{\alpha_{\ell+1}}) \bigr ] \\
      & = \frac{1}{2} \ \bigl [ \ \!\!\sum_{\alpha_{\ell+1}} \!\!
      (u_{\ell+1}^{\alpha_{\ell+1}} + \delta u_{\ell+1}^{\alpha_{\ell+1}}) \! \sum_{\alpha} \!\!
      a_{\alpha,\alpha_{\ell+1}} \! (u^{\alpha} + \delta u^{\alpha}) \ + \\ & \hspace*{6.1cm} + \
      \!\!\sum_{\alpha_{\ell+1}} \!\!
      (u_{\ell+1}^{\alpha_{\ell+1}}
      - \delta u_{\ell+1}^{\alpha_{\ell+1}}) \! \sum_{\alpha} \!\!
      a_{\alpha,\alpha_{\ell+1}} \! (u^{\alpha}
      - \delta u^{\alpha})  \bigr ],
\end{split}
\]
with
 $u^\alpha,\,u_{\ell+1}^{\alpha_{\ell+1}}\in\mathcal{P}(\Gamma_{+})$,
$\alpha\in\N^\ell$ and $\alpha_{\ell+1}\in\N$.

By (\ref{eq:PROPERTIE3}) we have
\[u_{l+1}^{\alpha_{l+1}} - \delta u_{l+1}^{\alpha_{l+1}} = R_1(u_{l+1})H_{\alpha_{l+1}}\bigl(R_0(u_{l+1}),(R_1(u_{l+1}))^2\bigr)\] and by the induction hypothesis \eqref{induction1} we can write
\[
 u_{\ell+1}^{\alpha_{\ell+1}}
 + \delta u_{\ell+1}^{\alpha_{\ell+1}}~=~h_{\alpha_{\ell+1}}\bigl(R_0(u_{l+1}),(R_1(u_{l+1}))^2\bigr)
\]
and
\[
 \sum_\alpha a_{\alpha,\alpha_{\ell+1}}
 (u^\alpha + \delta u^\alpha)~=~h_{\alpha,\alpha_{\ell+1}}\bigl(R_0(u_i),R_1(u_i)R_1(u_j), \ 1 \leq i,j \leq l\bigr),
\]
with $H_{\alpha_{\ell+1}}, h_{\alpha_{\ell+1}},$
$\,h_{\alpha,\alpha_{\ell+1}}$ polynomial functions. Then $f$ can be rewritten as
\begin{equation}
\label{eq:F}
 \begin{split} & \frac{1}{2} \ \bigl [ \ \!\!\sum_{\alpha_{\ell+1}} \!\!
 h_{\alpha_{\ell+1}}\bigl(R_0(u_{l+1}),(R_1(u_{l+1}))^2\bigr)
       h_{\alpha,\alpha_{\ell+1}}\bigl(R_0(u_i),R_1(u_i)R_1(u_j), \ 1 \leq i,j \leq l\bigr)\\ &  \hspace*{2.4cm} + \
      \!\!\sum_{\alpha_{\ell+1}} \!\!
      R_1(u_{l+1})H_{\alpha_{l+1}}\bigl(R_0(u_{l+1}),(R_1(u_{l+1}))^2\bigr) \! \sum_{\alpha} \!\!
      a_{\alpha,\alpha_{\ell+1}} \! (u^{\alpha}
      - \delta u^{\alpha})  \bigr ].
\end{split}
\end{equation}
Now we prove by induction on $l$ that for all $1 \leq i \leq l$ and $1 \leq k,m \leq j-1,$ $k \neq m,$ we have
\begin{equation}
\label{eq:DIFERENCE}
u^{\alpha} - \delta u^{\alpha} = \sum_{j=1}^l R_1(u_j)H_{\alpha, j}\bigl(R_0(u_i),(R_1(u_i))^2,R_1(u_k)R_1(u_m)\bigr),
\end{equation} where $H_{\alpha,j}: \R^n \rightarrow \R$ are polynomial functions, with $n = 2l + {(j-1)(j-2)}/{2}.$ In fact, for $l=1$ it is obvious by (\ref{eq:PROPERTIE3}). Assume then that
\begin{equation}
\label{induction2}
u^{\tilde{\alpha}} - \delta u^{\tilde{\alpha}} = \sum_{j=1}^{l-1} R_1(u_j)H_{\tilde{\alpha}, j}\bigl(R_0(u_i),(R_1(u_i))^2,R_1(u_k)R_1(u_m)\bigr),\end{equation} where $1 \leq i \leq l-1,$ $1 \leq k,m \leq j-1,$ $k \neq m,$ with $\tilde{\alpha}= (\alpha_1, \ldots,
\alpha_{l-1})$ and $u^{\tilde{\alpha}}~=~u_1^{\alpha_1}\ldots u_{l-1}^{\alpha_{l-1}}.$ We write
\[u^{\alpha} - \delta u^{\alpha} = \frac{1}{2} \bigl[(u^{\tilde{\alpha}} - \delta u^{\tilde{\alpha}})(u_l^{\alpha_l} + \delta u_l^{\alpha_l}) + (u^{\tilde{\alpha}} + \delta u^{\tilde{\alpha}})(u_l^{\alpha_l} - \delta u_l^{\alpha_l}) \bigr].\] By the induction hypotheses (\ref{induction1}) and (\ref{induction2}), we have $u^{\alpha} - \delta u^{\alpha}$ given by
\[\begin{split}  & \frac{1}{2} \bigl[ \sum_{j=1}^{l-1}R_1(u_j)H_{\tilde{\alpha}, j}\bigl(R_0(u_i),(R_1(u_i))^2,R_1(u_k)R_1(u_m)\bigr)h_{\alpha_l}\bigl(R_0(u_l),(R_1(u_l))^2\bigr) \\
& \hspace*{3.2cm} + h_{\tilde{\alpha_l}}\bigl(R_0(u_i),R_1(u_i)R_1(u_n)\bigr)R_1(u_l)H_{\alpha_l}\bigl(R_0(u_l),(R_1(u_l))^2\bigr) \bigr], \end{split} \]
where $1 \leq i,n \leq l-1$ and $1 \leq k,m \leq j-1,$ with $k \neq m.$ Thus \[ \begin{split} u^{\alpha} - \delta u^{\alpha} & = \bigl[ \sum_{j=1}^{l-1}R_1(u_j)G_{\alpha,j}\bigl(R_0(u_i),(R_1(u_i))^2,R_1(u_k)R_1(u_m)\bigr) \\ & \hspace*{3.7cm} + R_1(u_l) G_{\alpha_l}\bigl(R_0(u_i),(R_1(u_i))^2,R_1(u_p)R_1(u_q)\bigr) \bigr], \end{split} \]
where $1 \leq i \leq l,$ $1 \leq k,m \leq j-1,$ with $k \neq m,$ and $1 \leq p,q \leq l-1,$ with $p \neq q,$ which completes the proof of (\ref{eq:DIFERENCE}). Therefore
\[
 \begin{split}
 \sum_{\alpha} \!\!
      a_{\alpha,\alpha_{\ell+1}} \! (u^{\alpha} - \delta u^{\alpha}) & = \\ & \hspace*{-.6cm} = \sum_{j=1}^l R_1(u_j) \ \sum_{\alpha} \!\!
      a_{\alpha,\alpha_{\ell+1}} H_{\alpha, j}\bigl(R_0(u_i),(R_1(u_i))^2,R_1(u_k)R_1(u_m)\bigr) \\
      & \hspace*{-.6cm} = \sum_{j=1}^l R_1(u_j) H_j\bigl(R_0(u_i),(R_1(u_i))^2,R_1(u_k)R_1(u_m)\bigr),
      \end{split}\] with $1 \leq i \leq l,$ $1 \leq k,m \leq j-1,$ $k \neq m.$ Returning to (\ref{eq:F}) we obtain
\[
 \begin{split}
f
 & = H_{\alpha,\alpha_{\ell+1}}\bigl(R_0(u_i),(R_1(u_i))^2,R_1(u_k)R_1(u_m)\bigr)\\ & \hspace*{2.4cm} +
      \sum_{j=1}^l R_1(u_j)R_1(u_{l+1}) H_{j,\alpha_{l+1}}\bigl(R_0(u_i),(R_1(u_i))^2,R_1(u_p)R_1(u_q)\bigr) \bigr ],
\end{split}
\] with  $1 \leq i \leq l+1,$ $1 \leq k,m \leq l,$ $ k \neq m,$ and $1 \leq p,q \leq j-1,$ $p \neq q.$ Therefore
\[f = H\bigl(R_0(u_i), R_1(u_i)R_1(u_j), 1 \leq i,j \leq l+1\bigr),\]  as desired. \qed
\end{proof}\\

\begin{example}[$\Gamma=$ {\bf O}(2)] \rm
Consider the orthogonal group {\bf O}(2) acting on $\C \times \R$,
where the rotations $\theta \in {\bf SO}(2)$ and the flip $\kappa$
act as
\[\theta(z,x)~=~(e^{i \theta} z, x)
 \qquad\text{and}\qquad \kappa (z,x)~=(~\bar{z}, -x). \]
Consider $\Gamma_{+} = {\bf SO}(2)$ and $\delta = \kappa \in
\Gamma_{-}$. We have that $u_1(z,x)=z\bar{z}$ and $u_2(z,x)=x$
form a Hilbert basis of $\mathcal{P}(\SOtwo)$. We have
\[R_0(u_1)(z,x)~=~z\bar{z}, \ \ R_1(u_2)(z,x)~=~x, \ \ R_0(u_2)(z,x) = R_1(u_1)(z,x) = 0.\]
By Theorem \ref{thm:MAIN1}, the ring $\mathcal{P}({\bf O}(2))$ is generated by
\[ R_0(u_1)(z,x)~=~z\bar{z} \quad \text{and} \quad R_1^2(u_2)(z,x)~=~x^2. \qquad \END\]
\end{example}

\vspace*{.4cm}

\begin{example}[$\Gamma=({\mathbf{D}_6}  \rtimes  \Ttwo) \oplus \Ztwo$] \rm
Consider the group $\Gamma=({\mathbf{D}_6} \rtimes  \Ttwo) \oplus
\Ztwo$ acting on $\C^3$ as follows: for $(z_1, z_2, z_3) \in
\C^3$,

\begin{itemize}
\item[(a)] the permutations in ${\mathbf{D}_6}$ permute the coordinates
of $(z_1,z_2,z_3),$
\item[(b)] the generator flip in ${\mathbf{D}_6}$ acts as $(z_1,z_2,z_3)
\mapsto (\bar{z}_1,\bar{z}_2,\bar{z}_3),$
\item[(c)] $\theta = (\theta_1, \theta_2)$ in the torus $\Ttwo$ acts as
\[\theta \cdot (z_1,z_2,z_3) =
(e^{i\theta_1p}z_1,e^{i\theta_2p}z_2,e^{-i(\theta_1 +
\theta_2)p}z_3),\]
\item[(d)] the  reflection $\kappa \in \Ztwo$ acts as minus the
identity: \[\kappa(z_1,z_2,z_3) = (-z_1,-z_2,-z_3).\]
\end{itemize}

In this example we consider $\delta = \kappa \in \Gamma_{-}$ and
so $\Gamma_{+} = {\mathbf{D}_6} \rtimes \Ttwo.$

Let $v_j = z_j\bar{z}_j,$ $j = 1,2,3$ and consider the elementary
symmetric polynomials in $v_j:$
\[u_1 = v_1 + v_2 + v_3, \quad u_2 =
v_1v_2 + v_1v_3 + v_2v_3, \quad u_3 =  v_1 v_2 v_3.\] Also, let
$u_4 = z_1z_2z_3 + \bar{z}_1\bar{z}_2\bar{z}_3$. It is well known
that $\{u_1, \dots, u_4\}$ is a Hilbert basis for
$\mathcal{P}_{\C^3}({\mathbf{D}_6} \rtimes \Ttwo)$ (see
Golubitsky \etal~\cite[Theorem 3.1(a), p. 156]{GS69}).

We have $R_0(u_j)=u_j$ and $R_1(u_j)=0,$ for $j=1,2,3$, $R_0(u_4)=0$ and $R_1(u_4)=u_4.$ By Theorem \ref{thm:MAIN1}, the ring $\mathcal{P}(\Gamma)$ is generated by $u_1,$ $u_2,$ $u_3$ and $u_4^2.$  \END
\end{example}

The following theorem is an extension of Theorem \ref{thm:MAIN1} in the computation of $\Gamma$-invariant generators for index $m$ greater than 2. The basic idea here is to take into account that $f \in \mathcal{P}(\Gamma)$ if, and only if, $R_0(f)=f$ and then to deduce the invariant generators under the action of $\Gamma$ from a Hilbert basis of the ring of invariants under the action of the subgroup $H.$

\begin{theorem} \label{thm:MAIN2}
Let $H$ be the kernel of an epimorphism $\sigma$ as in (\ref{eq:HOMOSIGMA}). Let $\{u_1,\ldots,u_s\}$ be a
Hilbert basis of the ring $\mathcal{P}(H)$. For each $j \in \{1, \ldots, m-1\},$ set $\alpha(j) = \sum_{i=1}^{s} \alpha_{ji},$ with $\alpha_{ji} \in \N.$ Then the invariants \[R_0(u_1), R_0(u_2), \ldots, R_0(u_s),\]
\[R_1(u_1)^{\alpha_{11}} \ldots R_1(u_s)^{\alpha_{1s}} \ldots R_{m-1}(u_1)^{\alpha_{m-1, 1}} \ldots R_{m-1}(u_s)^{\alpha_{m-1,s}},\]

\noindent such that $\displaystyle \sum_{j=1}^{m-1}j \alpha(j) \equiv 0 \ \text{mod} \ m,$ form a Hilbert basis of the ring $\mathcal{P}(\Gamma).$

\end{theorem}

\begin{proof}
Let us fix $\delta \in \Gamma$ such that $\sigma(\delta)$ is the primitive $m$-th root of unity  and let $\{u_1,\dots,u_s\}$ be a
Hilbert basis of $\mathcal{P}(H)$. Let $f \in \mathcal{P}(\Gamma).$ From
Proposition~\ref{prop:PROPERTIES1}, $R_0(f) = f,$ which gives
\begin{equation}
\label{eq:R0}
f(x) = \frac{1}{m-1} \sum_{k=1}^{m-1}f(\delta^k x), \quad \forall \ x \in V.
\end{equation} For each $i \in \{1, \ldots, s\}$ we have
$u_i = \displaystyle \sum_{j=0}^{m-1}R_j(u_i).$ So \[\{R_0(u_i), R_1(u_i), \ldots, R_{m-1}(u_i): 1 \leq i \leq s\}\] is a new Hilbert basis of $\mathcal{P}(H).$ Using multi-index notation, every polynomial function $f \in\mathcal{P}(\Gamma)$ can be written as
\[f~=~\sum a_{\alpha} R_0(u_i)^{\alpha_{\!0}} \,
 R_{1}(u_i)^{\alpha_{\!1}} \ldots R_{m-1}(u_i)^{\alpha_{m-1}}~,\]
 with $a_{\alpha} \in \C,$ where $\alpha = (\alpha_{\!0},\ldots,\alpha_{m-1}),$ $\alpha_j = (\alpha_{j\!1},\ldots,\alpha_{js}) \in \N^s$ and $R_j(u_i)^{\alpha_{\!j}} = R_{j}(u_1)^{\alpha_{j\!1}} \ldots R_{j}(u_s)^{\alpha_{js}},$ for $j \in \J.$

Now, for every $k \in \N$ and $x \in V,$ we have \[R_j(u_i)(\delta^k x) = \sigma^j(\delta^k)R_j(u_i)(x) = \sigma^{jk}(\delta)R_j(u_i)(x),\] from which we obtain
\begin{equation}
\label{eq:GENERAL}
 f(\delta^k x)~=~\sum \sigma(\delta)^{k \sum_{j=1}^{m-1} j \alpha(j)} a_{\alpha} \,R_0(u_i)^{\alpha_{\!0}} \,
 R_{1}(u_i)^{\alpha_{\!1}} \ldots R_{m-1}(u_i)^{\alpha_{m-1}}(x),
\end{equation} where $\alpha(j) = \alpha_{j1} + \ldots + \alpha_{js},$ for all $1 \leq j \leq m-1.$

We now use (\ref{eq:GENERAL}) in (\ref{eq:R0}) to get
\[
 \sigma^N(\delta) + \sigma^{2N}(\delta) + \ldots + \sigma^{(m-1)N}(\delta) = m-1,\] where $N= \displaystyle \sum_{j=1}^{m-1} j \alpha(j) \in \N.$ It follows from
 (\ref{eq:ROOTUNITY}) that $\sigma^N(\delta) = 1,$ ie, $N \equiv$ 0 mod $m.$ Therefore the generators of $\mathcal{P}(\Gamma),$ as a ring, are given by \[R_0(u_1), R_0(u_2), \ldots, R_0(u_s)\] and by the products
 \[R_1(u_1)^{\alpha_{11}} \ldots R_1(u_s)^{\alpha_{1s}} \ldots R_{m-1}(u_1)^{\alpha_{m-1,1}} \ldots R_{m-1}(u_s)^{\alpha_{m-1,s}},
\] for $\displaystyle \sum_{j=1}^{m-1}j \alpha(j) \equiv 0$ mod $m.$ \qed
\end{proof}

\vspace*{.2cm}

\begin{example}[$\Gamma= \Zthree \times \Zthree$] \rm
Consider the group $\Zthree \times \Zthree$ acting on $V=\C^2$,
where $\xi = (\xi_1,\xi_2) \in \Gamma$ acts on $z=(z_1,z_2) \in V$  as
\[\xi(z)~=~(\xi_1 z_1,\xi_2 z_2).\] Fix $\sigma: \Gamma \rightarrow \Zthree$ such that $H = \ker \sigma = \Zthree \times \mathrm{1},$ and set $\delta = (1,e^{\frac{2\pi i}{3}}) \in \Gamma \setminus H$.
The functions $u_i,$ $1 \leq i \leq 5,$ given by \[u_1(z_1,\bar{z_1},z_2,\bar{z_2})=z_1\bar{z}_1, \quad u_2(z_1,\bar{z_1},z_2,\bar{z_2})=z_1^3, \quad u_3(z_1,\bar{z_1},z_2,\bar{z_2})=\bar{z}_1^3,\] \[u_4(z_1,\bar{z_1},z_2,\bar{z_2})=z_2, \quad u_5(z_1,\bar{z_1},z_2,\bar{z_2})=\bar{z}_2\]
form a Hilbert basis of $\mathcal{P}(H)$. In this case, we have
$R_0(u_i) = u_i,$ for $\ 1 \leq i \leq 3,$ $R_1(u_4) = u_4$ and $R_2(u_5) = u_5.$ By  Theorem \ref{thm:MAIN2}, the generators of the ring $\mathcal{P}(\Gamma)$ are \[R_0(u_1),\ R_0(u_2),\ R_0(u_3),\ R_0(u_4),\ R_0(u_5), \ R_1(u_i)R_1(u_j)R_1(u_k),\] \[ R_2(u_i)R_2(u_j)R_2(u_k), \ R_1(u_i)R_2(u_j),\] for all $1 \leq i, j, k \leq 5.$ Therefore the ring $\mathcal{P}(\Zthree \times \Zthree)$ is generated by the elements
\[ z_1\bar{z}_1, \ z_1^3, \ \bar{z}_1^3, \ z_2\bar{z}_2, \ z_2^3, \ \bar{z}_2^3.\] \END

\end{example}

\end{document}